\documentclass{amsart}

\usepackage{amssymb}
\usepackage[cp1250]{inputenc}
\usepackage{picture}
\usepackage{color}
\usepackage{tikz}
\usepackage{tikz}
\usetikzlibrary{matrix}

\newtheorem{thm}{Theorem}[section]
\newtheorem{lem}[thm]{Lemma}
\newtheorem{exa}[thm]{Example}

\newtheorem{cor}[thm]{Corollary}

\author{Jacek Marchwicki}
\address{
Faculty of Mathematics and Computer Science,
University of Warmia and Mazury in Olsztyn,
S\l oneczna 54,
10-710 Olsztyn,
Poland}
\email {jacek.marchwicki@uwm.edu.pl}

\title[Semi-fast convergent representations of purely atomic finite measures]{Semi-fast convergent representations of purely atomic finite measures}


\begin{document}

\begin{abstract}
The paper is dedicated to calculate the number of semi-fast convergent representations of a set of subsums. The formula for specific irreductible sequences is obtained. We study the Central Cantor sets and consider complete sequences of natural numbers. 
\end{abstract}

\maketitle 

\section{Introduction}
Let $(x_n)$ be a summable sequence. By the achievement set of the sequence $(x_n)$ we mean the set of all subsums of the series $\sum_{n=1}^{\infty}x_n$, that is
$$
E(x_n)=\{\sum_{n\in A}x_n : A\subset\mathbb{N}\}=\{\sum_{n=1}^\infty\varepsilon_nx_n:(\varepsilon_n)\in\{0,1\}^\mathbb{N}\}.
$$ 
In the whole paper we assume that the considered sequences $(x_n)$ are non-increasing and summable. 

Let $\mu$ be a purely atomic finite measure defined on $\mathbb{N}$ and $n$ be its $n$-th largest atom. If we put $x_n=\mu(\{n\})$ for all $n$ then the range of $\mu$ - $rng(\mu)$ is equal to the achievement set $E(x_n)$. 

The first paper devoted to the sets of subsums was \cite{Kakeya}, where the Author proved that 
\begin{thm}\label{Kakeya}
The following assertions hold
\begin{itemize}
\item[(i)] $E(x_n)$ is a compact perfect or finite set,
\item[(ii)] If $x_{n}>\sum_{i>n}x_{i}$ for all sufficiently large $n$'s,
then $E(x_{n})$ is homeomorphic to the ternary Cantor set $C$,
\item[(iii)] The inequality $x_{n}\leq \sum_{i>n}x_{i}$ holds for all sufficiently large $%
n $'s if and only if $E(x_{n})$ is a finite union of closed intervals. Moreover the inequality $x_n\leq\sum_{i>n}x_i$ is satisfied for all $n$ iff $E(x_n)$ is an interval. 
\end{itemize}
\end{thm}
If the conditions $(ii)$ either $(iii)$ holds for each $n$ then the sequence $(x_n)$ (also the series $\sum_{n=1}^{\infty}x_n$) is called fast or slow convergent respectively.  

Guthrie and Nymann \cite{GN} proved that there is one more possible form of the set of subsums - so called Cantorval, which is homeomorphic to the set 
\begin{equation*}
[0,1]\setminus \bigcup_{n\in \mathbb{N}}U_{2n},
\end{equation*}%
where $U_{n}$ denotes the union of $2^{n-1}$ open middle thirds which are
removed from the interval $[0,1]$ at the $n$-th step in the construction of
the classic Cantor ternary set $C$. 

\begin{figure}[h]
\centering

\begin{center}
\begin{tikzpicture}[x=14cm,y=1.3cm]

\draw[line width=2pt] (0,0) -- (1,0);
\node[below] at (0,0) {$0$};
\node[below] at (1,0) {$1$};

\draw[line width=2pt] (0,-1) -- (1/3,-1);
\draw[line width=2pt,red] (1/3,-1) -- (2/3,-1);
\draw[line width=2pt] (2/3,-1) -- (1,-1);

\node[below] at (0,-1) {$0$};
\node[below] at (1/3,-1) {$\frac13$};
\node[below] at (2/3,-1) {$\frac23$};
\node[below] at (1,-1) {$1$};

\draw[line width=2pt] (0,-2) -- (1/9,-2);
\draw[line width=2pt] (2/9,-2) -- (1/3,-2);

\draw[line width=2pt,red] (1/3,-2) -- (2/3,-2);

\draw[line width=2pt] (2/3,-2) -- (7/9,-2);
\draw[line width=2pt] (8/9,-2) -- (1,-2);

\node[below] at (0,-2) {$0$};
\node[below] at (1/9,-2) {$\frac19$};
\node[below] at (2/9,-2) {$\frac29$};
\node[below] at (1/3,-2) {$\frac13$};
\node[below] at (2/3,-2) {$\frac23$};
\node[below] at (7/9,-2) {$\frac79$};
\node[below] at (8/9,-2) {$\frac89$};
\node[below] at (1,-2) {$1$};


\draw[line width=2pt] (0,-3) -- (1/27,-3);
\draw[line width=2pt,red] (1/27,-3) -- (2/27,-3);
\draw[line width=2pt] (2/27,-3) -- (1/9,-3);

\draw[line width=2pt] (2/9,-3) -- (7/27,-3);
\draw[line width=2pt,red] (7/27,-3) -- (8/27,-3);
\draw[line width=2pt] (8/27,-3) -- (1/3,-3);

\draw[line width=2pt,red] (1/3,-3) -- (2/3,-3);

\draw[line width=2pt] (2/3,-3) -- (19/27,-3);
\draw[line width=2pt,red] (19/27,-3) -- (20/27,-3);
\draw[line width=2pt] (20/27,-3) -- (7/9,-3);

\draw[line width=2pt] (8/9,-3) -- (25/27,-3);
\draw[line width=2pt,red] (25/27,-3) -- (26/27,-3);
\draw[line width=2pt] (26/27,-3) -- (1,-3);

\node[below] at (0,-3) {$0$};
\node[below] at (1/9,-3) {$\frac19$};
\node[below] at (2/9,-3) {$\frac29$};
\node[below] at (1/3,-3) {$\frac13$};
\node[below] at (2/3,-3) {$\frac23$};
\node[below] at (7/9,-3) {$\frac79$};
\node[below] at (8/9,-3) {$\frac89$};
\node[below] at (1,-3) {$1$};


\draw[line width=2pt] (0,-4) -- (1/81,-4);
\draw[line width=2pt] (2/81,-4) -- (1/27,-4);

\draw[line width=2pt] (2/27,-4) -- (7/81,-4);
\draw[line width=2pt] (8/81,-4) -- (1/9,-4);

\draw[line width=2pt] (2/9,-4) -- (19/81,-4);
\draw[line width=2pt] (20/81,-4) -- (7/27,-4);

\draw[line width=2pt] (8/27,-4) -- (25/81,-4);
\draw[line width=2pt] (26/81,-4) -- (1/3,-4);

\draw[line width=2pt,red] (1/27,-4) -- (2/27,-4);
\draw[line width=2pt,red] (7/27,-4) -- (8/27,-4);
\draw[line width=2pt,red] (1/3,-4) -- (2/3,-4);
\draw[line width=2pt,red] (19/27,-4) -- (20/27,-4);
\draw[line width=2pt,red] (25/27,-4) -- (26/27,-4);

\draw[line width=2pt] (2/3,-4) -- (55/81,-4);
\draw[line width=2pt] (56/81,-4) -- (19/27,-4);

\draw[line width=2pt] (20/27,-4) -- (61/81,-4);
\draw[line width=2pt] (62/81,-4) -- (7/9,-4);

\draw[line width=2pt] (8/9,-4) -- (73/81,-4);
\draw[line width=2pt] (74/81,-4) -- (25/27,-4);

\draw[line width=2pt] (26/27,-4) -- (79/81,-4);
\draw[line width=2pt] (80/81,-4) -- (1,-4);

\node[below] at (0,-4) {$0$};
\node[below] at (1/9,-4) {$\frac19$};
\node[below] at (2/9,-4) {$\frac29$};
\node[below] at (1/3,-4) {$\frac13$};
\node[below] at (2/3,-4) {$\frac23$};
\node[below] at (7/9,-4) {$\frac79$};
\node[below] at (8/9,-4) {$\frac89$};
\node[below] at (1,-4) {$1$};

\end{tikzpicture}
\end{center}

\caption{Consecutive iterations in the construction of a Cantorval.
The red intervals mark the middle thirds that are removed in the classical construction of the ternary Cantor set.}
\end{figure}
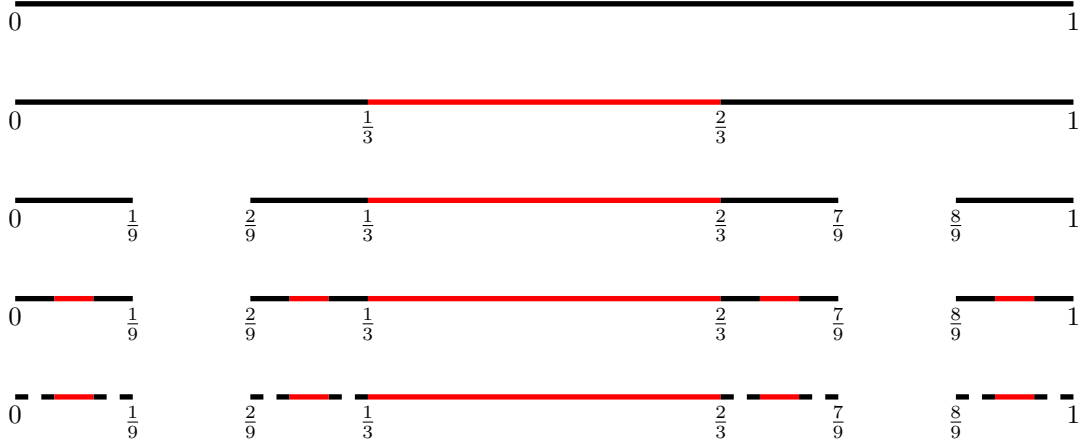

It is known that a Cantorval is just a nonempty compact
set in $\mathbb{R}$, that it is the closure of its interior and both
endpoints of any nontrivial component are accumulation points of its trivial
components. Other topological characterizations of Cantorvals can be found
in \cite{BFPW} and \cite{MO}.

We will consider the following subsets of the achievement set $E(x_n)$: the countable set $F(x_n)$ of all finite subsums, $F_k(x_n)$- the finite set of subsums of $k$-initial terms and $E_k(x_n)$ by which we denote the achievement set for the tail $(x_n)_{n>k}$, that is 
$$F(x_n)=\{\sum_{n\in E}x_n :E\subset\mathbb{N}, \vert E\vert <\infty\};$$
$$F_k(x_n)=\{\sum_{n=1}^{k}\varepsilon_nx_n: (\varepsilon_n)\in\{0,1\}^{k}\};$$
$$E_k(x_n)=\{\sum_{n=k+1}^\infty\varepsilon_nx_n:(\varepsilon_n)\in\{0,1\}^\mathbb{N}\}.$$
For any $k$ we have the following decomposition of the achievement set
$$E(x_n)=F_k(x_n)+E_k(x_n).$$
We also denote $r_k=\sum_{n>k}x_n$ for each $k$, that is the sum of the $k$-th tail of the series.

It is clear that not every finite union of closed intervals, Cantor set or Cantorval is a set of subsums of some sequence. Clearly any achievement set $E(x_n)$ contains zero and is symmetric in the sense that there exists a number $t$
such that if $t-x\in E(x_n)$ then $t+x\in E(x_n)$ too. However the symmetry is not only the case and deeper studies of achievable sets, that is that which are achievement sets for some sequence, can be found in \cite{recover}.  

We study whether a given achievable set $E$ is obtained as a set of subsum for a unique non-increasing sequence of positive terms.  If the answer is negative we study the problem of number of different sequences with the same achievement set. For a given achievable set $E$ by its representation we call any sequence $(x_n)$ such that $E:=E(x_n)$. 
Let us start with an example \cite{recover}.
\begin{exa}\label{dwarozlaczne}\emph{
Let us consider the unite interval $E:=[0,1]$. Probably the simplest and the most popular is its binary representation $E(x_n)=[0,1]$ for  $x_n=\frac{1}{2^n}$.  By Steinhaus Theorem \cite{Steinhaus} we also know that the sum of two ternary Cantor sets $C+C$ is equal to $[0,2]$, so  $E(y_n)=[0,1]$, where $y_{2n-1}=\frac{1}{3^n}=y_{2n}$. One can easily observe, that the both sequences $(x_n)$ and $(y_n)$ are slowly convergent and $\sum_{n=1}^{\infty} x_n=\sum_{n=1}^{\infty} y_n=1$. Moreover $\{x_n : n\in\mathbb{N}\}\cap\{y_n : n\in\mathbb{N}\}=\emptyset$.}
\end{exa}
In Example \ref{dwarozlaczne} we showed that there exists two slowly convergent series with the same sum and no common terms. The following holds for interval-like achievement sets \cite{recover}.
\begin{thm}\label{intervalcharacterization}
Let $E$ be an achievable set. The following conditions are equivalent. 
\begin{itemize}
\item[(i)] $E$ is an interval;
\item[(ii)] there exist two sequences $(x_n)$ and $(y_n)$ such that $E=E(x_n)=E(y_n)$ and $\{x_n : n\in\mathbb{N}\}\cap\{y_n : n\in\mathbb{N}\}=\emptyset$;
\item[(iii)] for any sequence $(x_n)$ with $E=E(x_n)$, there exists $(y_n)$ such that $E=E(y_n)$ and  $\{x_n : n\in\mathbb{N}\}\cap\{y_n : n\in\mathbb{N}\}=\emptyset$;
\item[(iv)]  for any sequence $(x_n)$ with $E=E(x_n)$, there exists $(y_n)$ such that $E=E(y_n)$ and  $F(x_n)\cap F(y_n)=\emptyset$.
\end{itemize}
\end{thm}

By a gap in $E\subset\mathbb{R}$ we mean an open interval $(a,b)$ such that $a,b\in E$ and $(a,b)\cap E=\emptyset$. We use the following result proved in \cite{BFGSW}. 
\begin{lem}\label{thirdgaplemma}
(Third Gap Lemma) Suppose that $(a,b)$ is a gap in the range $E(x_n)$ such that for
any gap $\left( a_{1},b_{1}\right) $ with $b_{1}<a$ we have $b-a>b_{1}-a_{1}$
(in other words $\left( a,b\right) $ is the longest gap from the left). Then for some $k\in \mathbb{N}$  we have $b=x_{k}$  and $a=r_{k}$.
\end{lem}

Hence if we have interval-like achievement set $E(x_n)$ then we cannot recover the sequence $(x_n)$. We can not even get a one term of the sequence.  The case is different when $E(x_n)$ is a finite sum of closed intervals. 
By Theorem \ref{Kakeya} (iii) there exists $k$ such that the inequality $x_n\leq \sum_{i=n+1}^{\infty}x_i$ holds for all $n>k$ if and only if $E(x_n)$ is a finite union of closed intervals. More precisely $$E(x_n)=F_k(x_n)+E_k(x_n)=\{\sum_{i=1}^{k}\varepsilon_i x_i : (\varepsilon_i)_{i=1}^{k}\in\{0,1\}^k\}+E\big((x_n)_{n>k}\big),$$ where $E\big((x_n)_{n>k}\big)=[0,\sum_{n>k}x_n]$. Let $F_k(x_n)=\{0=f_0<f_1<\ldots <f_m=\sum_{i=1}^{k}x_i\}$. Note that if $f_{i}-f_{i-1}>r_k$ then $(f_{i-1}+r_k,f_i)$ is a gap in $E(x_n)$. Moreover if the gap is the longest from the left then $f_i$ is one of the first $k$-terms of the sequences $(x_n)$. In particular for any two representations $(x_n)$ and $(y_n)$ we get that $f_i$ is its common term.

Conversely, if $E(x_n)$ is not a multi-interval then it has infitely many gaps among which are infinitely many longest from the left. Thus $\{x_n : n\in\mathbb{N}\}\cap\{y_n : n\in\mathbb{N}\}=\emptyset$ is infinite. Hence we get 
\begin{thm}
The achievement set $E$ is a finite union of intervals if and only if there exist two representations $(x_n)$ and $(y_n)$ such that  the intersection $\{x_n : n\in\mathbb{N}\}\cap\{y_n : n\in\mathbb{N}\}=\emptyset$ is finite. 
\end{thm}

To sum up, one can construct two interval-filling sequences with the same sum and no common terms. Conversely, if the set of subsums is a fnite disjoint union of at least two closed intervals then any its representation need to have specific terms. Some of them are connected with the specific structure of the achievement set while the others can be  obtained by the general methods like the Third Gap Lemma. The case is different when the achievement set is a Cantor set.
 
\begin{exa}
Let us consider the ternary Cantor set $C=E(x_n)$, where  $x_n=\frac{2}{3^n}$. Observe that any of the numbers $x_{n}=\frac{2}{3^n}$ is the right end of the longest gaps  from the left of $C$. Indeed, $r_n=\frac{1}{3^n}$ for all $n$, so any gap $(r_n,x_n)$ is as long as  part of the achievement set $C\cap [0,r_n]$ remaining on its left. 

Suppose that  $E(y_n)=C$ for some sequence $(y_n)$. Then $\{x_n:n\in\mathbb{N}\}\subset \{y_n:n\in\mathbb{N}\}$. Observe that $\sum_{n=1}^{\infty} x_n=1$,  so $\sum_{n=1}^{\infty} y_n\geq 1$. But $\sum_{n=1}^{\infty} y_n=1$, so we cannot add any more terms to the sequence $(y_n)$. Hence $y_n=x_n$ for all $n$ which means that the ternary Cantor set is obtained in the unique way as achievement set of non-increasing sequence by the sequence $(x_n)$.
\end{exa}

The Authors in \cite{recover} formulated a general result. 
\begin{thm}\label{najdluzszeluki}
Assume that $(x_n)$ is fast convergent.
If for any $n$ the gap $(r_n,x_n)$ is the longest from the left and $E(y_n)=E(x_n)$ then $y_n=x_n$, that is the achievement set has a unique representation. 
\end{thm} 

Although the fast convergence of the sequence $(x_n)$ guarantees that the inequality $x_n>r_n$ holds for any $n$, which means that $(r_n,x_n)$ is a gap but we can not conclude that it satisfies the assumptions of Third Gap Lemma. Indeed let us consider the following Example.

\begin{exa}\label{szybkozbieganieluki}
Let $x_1=5$, $x_2=3$, $x_{n+2}=\frac{2}{3^n}$ for all $n$. Then $x_1>r_1=4$ and $x_2>r_2=1$, so the sequence $(x_n)$ is fast convergent. However $(r_1,x_1)=(4,5)$ is not longer than the gap $(r_2,x_2)=(1,3)$ on its left. 
\end{exa}

Hence we need some stronger condition than fast convergence to be sufficient for satisfying the assumptions of Theorem  \ref{najdluzszeluki}. It is clear that if $x_n\geq 2r_n$ then any gap $(r_n,x_n)$ is the longest from the left since it is even longer than the rest part of the achievement set on its left. Hence we immediately obtain the uniqueness of representation for geometric set of subsums.

\begin{exa}
Let $q\leq\frac{1}{3}$. The set $E(q^n)$ has unique representation by geometric sequence. Indeed, we have 
$$x_n=q^n\geq \frac{2q^{n+1}}{1-q}=2r_n \Leftrightarrow q\leq\frac{1}{3}.$$
\end{exa}

However it appears that we have a weaker sufficient condition \cite{recover}.

\begin{thm}\label{dwarazywiekszy}
Assume that $x_n>2x_{n+1}$ for every $n\in\mathbb{N}$. Then $(x_n)$ satisfies the assumptions of Theorem \ref{najdluzszeluki}.
\end{thm}


Hence we get the stronger result about geometric achievement sets. 
\begin{exa}
Let $q<\frac{1}{2}$. The set $E(q^n)$ has unique representation by geometric sequence. 
\end{exa}

Theorem \ref{dwarazywiekszy} can be also used to obtain uniquely represented family of Cantor sets with positive Lebesque measure.
\begin{exa}
Let $q\in (0,\frac{1}{2})$ and $x_n=\frac{1}{2^n}+q^n$ for $n\in\mathbb{N}$. Then a sequence $(x_n)$ satisfies a condition given in Theorem \ref{dwarazywiekszy}, so the set $E(x_n)$ is the achievement set of the only one sequence. Moreover the Lebesgue measure of the set $E(x_n)$ can be calculated by the formula given in \cite{BFPW}, namely $$\lambda\big(E(x_n)\big)=\lim_{n\to\infty}2^n r_n=\lim_{n\to\infty}2^n(\frac{1}{2^n}+\frac{q^{n+1}}{1-q})=1.$$ 
\end{exa}

The sufficient condition $x_n>2x_{n+1}$ for $n\in\mathbb{N}$ given in Theorem \ref{dwarazywiekszy} is optimal for having unique representation of the set of subsums. Indeed if we assume even a slightly weaker inequality that $x_n\geq 2x_{n+1}$  for all $n\in\mathbb{N}$ then we may obtain various representations of the achievement set $E(x_n)$. The obvious counterexample is the interval filling sequence $x_n=\frac{1}{2^n}$ which leads to the binary expansions of the unite interval. However we can also find it among the fast convergent representations of Cantor sets.
\begin{exa}\label{dwarazynieostro}
Let us consider the multigeometric sequence defined as $x_{2n-1}=\frac{2}{5^n}, x_{2n}=\frac{1}{5^n}$ for each $n\in\mathbb{N}$ or shorter $(x_n)=(2,1;\frac{1}{5})$, see \cite{BFS}. Observe that $x_n>r_n$ and  $x_n\geq 2x_{n+1}$ for each $n\in\mathbb{N}$, so the series $\sum_{n=1}^{\infty} x_n$ is quickly convergent, but the condition $x_n>2x_{n+1}$ is satisfied only for even $n$'s. Define $y_{3n-2}=y_{3n-1}=y_{3n}=\frac{1}{5^n}$ or $(y_n)=(1,1,1;\frac{1}{5})$. Then we have $E(x_n)=E(y_n)$. We can construct even continuum many different sequences with the given set of subsums. Indeed let us define $(z_n^{(n_i)})$ in the following way: 
\\In the $k$-th step of construction we either define two elements $z_{n_k}^{(n_i)}=\frac{2}{5^k}$,  $z_{n_k+1}^{(n_i)}=\frac{1}{5^k}$ or three $z_{n_k}^{(n_i)}=z_{n_k+1}^{(n_i)}=z_{n_k+2}^{(n_i)}=\frac{1}{5^k}$, where $n_1=1$ and $n_{k+1}-n_k$ is equal to the number of defined elements. For each sequence $(z_n^{(n_i)})$ we obtain that $E(z_n^{(n_i)})=E(x_n)$ and clearly there are continuum many of ways of taking the sequences $(n_i)$, which lead to a different sequences.
\end{exa}
In Example \ref{dwarazynieostro} the gaps $(r_{2n},x_{2n})$ and $(r_{2n-1},x_{2n-1})$ have equal lengths. Indeed, the following equalities hold
$$x_{2n-1}-r_{2n-1}=\frac{2}{5^n}-x_{2n}-r_{2n}=\frac{1}{5^n}-r_{2n}=x_{2n}-r_{2n},$$
so the assumptions of Theorem \ref{najdluzszeluki} are not satisfied.

If we consider continuum many representations of the set $E(x_n)$ in Example \ref{dwarazynieostro} we can see that only the original sequence $(x_n)$ is fast convergent. Now we prove that if the Cantor set $E(x_n)$ has a fast convergent representation then it has to be unique. 

\begin{thm}
If the equality $E(x_n)=E(y_n)$ holds and both $(x_n)$ and $(y_n)$ are fast convergent then $x_n=y_n$ for all $n$. 
\begin{proof}
Since $r_1<x_1$ the following inequalities hold
$$2r_1<x_1+r_1=\sum_{n=1}^{\infty}x_n<2x_1,$$
which means that the point of reflection $\frac{1}{2}\sum_{n=1}^{\infty}x_n$ of  the set $E(x_n)$ belongs to a gap $(r_1,x_1)$. Note that for any alternative representation $(y_n)$ which is fast convergent the anaologous inequalities hold 
$$\sum_{n=2}^{\infty}y_n<\frac{1}{2}\sum_{n=1}^{\infty}y_n=\frac{1}{2}\sum_{n=1}^{\infty}x_n<y_1.$$ 
Hence $\frac{1}{2}\sum_{n=1}^{\infty}x_n$ belong to a gap $(\sum_{n=2}^{\infty}y_n,y_1)$. Thus $y_1=x_1$. In a similar way we prove that $y_k=x_k$ by using the achievement set $E_k(x_n)=E\big((x_n)_{n>k}\big)$  and the fact that the tail is fast convergent as well as the given sequence. 
\end{proof}
\end{thm}

\section{Semi-fast convergent sequences}


Now we consider a more general notion that fast convergence, which was introduced in \cite{BFPW2}. 
The sequence $(x_n)$ is semi-fast convergent if $x_n>\sum_{k: x_k<x_n}x_{k}$ holds for every $n$. As in the case of the fast convergent sequences achievement sets here are also Cantors. The alternative desciption for a semi-fast convergent sequence $(x_n)$ is as follows: there exists two uniquely determined sequences $(\alpha_k)$ of positive numbers decreasing to $0$ and $(N_k)$ of natural numbers such that $x_i=\alpha_k$ for $\sum_{j=0}^{k-1}N_j<i\leq \sum_{j=0}^{k-1}N_j$. The numbers $\alpha_k$ are the values of the terms of the sequence $(x_n)$ and $N_k$ is the multiplicity of the value $\alpha_k$. It means that 
$$(x_n)=(\underbrace{\alpha_1,\ldots,\alpha_1}_{N_1-\text{times}},\underbrace{\alpha_2,\ldots,\alpha_2}_{N_2-\text{times}},\ldots),$$
where $\alpha_k>\sum_{n>k}N_n\alpha_n$ for each $k$. We wil also denote semi-fast convergent sequence $(x_n)$ as $(\alpha_k,N_k)$.

The Authors \cite{BFPW2} proved that $E(\alpha_k,N_k)$ is a Central Cantor set (or equivalently has fast convergent representation) if and only if for any $k$ the following equality holds for the numbers of repeating terms $N_k=2^{m_k}-1$ for some natural number $m_k$. 

It is clear that if $m_k>1$, so $N_k\geq 3$ then the representation is nonunique. For instance, if $N_k=3$ then the terms $\alpha_k,\alpha_k,\alpha_k$ can be replaced with $2\alpha_k,\alpha_k$ and the whole achievement set $E(\alpha_k,N_k)$ does not change. In general we may consider the notion of complete sequence introduced in \cite{Brown}.

We say that a non-decreasing sequence of natural numbers $(z_n)$ is complete if $E(z_n)=\mathbb{N}\cup\{0\}$. The set of subsums for a sequence which does not have any subsequence tending to zero is reduced only to finite sums.  We can generalize this notion to finite sequence $(z_n)_{n=1}^{k}$ in the way that it is complete if $E\big((z_n)_{n=1}^{k}\big)=\{0,1,2,\ldots,\sum_{n=1}^{k}z_n\}$. 

For instance, the set $\{0,1,2,3,4,5,6\}$ is generated by the following five different complete sequences:
$$(1,1,1,1,1,1); \ \ (1,1,1,1,2); \ \ (1,1,2,2); \ \ (1,1,1,3); \ \ (1,2,3). $$

The alterative representation of $E(\alpha_k,N_k)$ can be constructed by making the following replacements on the segments of equal terms: \\
For a given $k$ we have $N_k$ repetitions of element $\alpha_k$. Take any non-constant complete sequence $(z_i)_{i=1}^{m}$, which generate the set $\{0,1,2,\ldots,N_k\}$ and replace the whole constant sequence of $\alpha_k$'s by $\beta_i=z_i\alpha_k$ for all $i\in\{1,\ldots,m\}$.

Of course the mentioned replacements can be done in any subset of $\mathbb{N}$ of indices $k$'s. Hence $E(\alpha_k,N_k)$ has continuum many representations if $N_k\geq 3$ for infinitely many $k$'s.

Now we would like to ask if there are any other representations than that in which we change the whole segment of constant elements by using some complete sequence.  The negative answer will appear in Theorem \ref{niemainnychrozwiniec}. 

Let $(X,d)$ be a metric space and $A\subset X$. We define the center of distances of set $A$ by 
$$C(A)=\{\alpha\geq 0: \forall_{x\in A}\exists_{y\in A} \ d(x,y)=\alpha\}.$$

We consider $X=\mathbb{R}$ equipped with standard topology generated by $d(x,y)=\vert y-x\vert$. 
In the paper \cite{BPW} the Authors proved the basic but very useful property of the center of distances of the set of subsums by giving the following two inclusions.
\begin{lem}\label{LematOCentrumDystansow}
The following inclusions hold $$\{x_n:n\in\mathbb{N}\}\subset C\big(E(x_n)\big)\subset E(x_n).$$
\end{lem}
In \cite{BBFP} the Authors calculated the center of distances of Central Cantor set represented by semi-fast convergent sequence.

\begin{thm}
Let $E(\alpha_i,N_i)$ be a central Cantor set given by semi-fast convergent sequences satisfying $\alpha_i>(N_{i+1}+1)\alpha_{i+1}$ for every $i$. Then $$C\big(E(\alpha_i,N_i)\big)=\{0\}\cup\{k\alpha_i : i\in\mathbb{N}, k\in\{1,2,\ldots,\frac{N_i+1}{2}\}\}.$$
\end{thm}

Hence by Lemma \ref{LematOCentrumDystansow} we know that in every representation  $(y_n)$ of $E(\alpha_i,N_i)$ we may use only the multiplies of the terms $\alpha_i$. Moreover observe that we have the following inequalities  $$\alpha_i>(N_{i+1}+1)\alpha_{i+1}>N_{i+1}\alpha_{i+1}+(N_{i+2}+1)\alpha_{i+2}>\ldots>\sum_{k=i+1}N_k\alpha_k.$$
Thus the inequality is a stronger condition than being semi-fast convergent.   
The reverse implication does not hold even for the fast convergent sequence as shows the following example.

\begin{exa}
Let $x_1=\alpha$ be any number in the interval $(1,\frac{4}{3})$  and $x_{n+1}=\frac{2}{3^n}$ for each $n$. Clearly $E(x_n)=C\cup (\alpha+C)$, where $C$ is the ternary Cantor set. The sequence is fast convergent but $x_1\leq 2x_2=(N_2+1)x_2$.

\end{exa}

The inequality $\alpha_i>N_{i+1}\alpha_{i+1}$ is important in considerations of semi-fast convergent sequences. 
Otherwise the two consequitive segments may coincide with each other \cite{BBFP}. 
A semi-fast convergent sequence $(\alpha_k,N_k)$  is irreductible if $\alpha_k\neq (N_{k+1}+1)\alpha_{k+1}$ for all $k$. 

If the sequence is reductible that is $\alpha_i =(N_{i+1}+1)\alpha_{i+1}$ for some $i$ then we can replace two segments $(\alpha_i,N_i)$ and $(\alpha_{i+1},N_{i+1})$ by the one segment $(\alpha_{i+1},N_i(N_{i+1}+1)+N_{i+1})$ to obtain another representation of the achievement set $E(\alpha_k,N_k)$. To omit that problems we study special case of irreductible sequences. We will prove the following. 

\begin{thm}\label{niemainnychrozwiniec}
Let $(\alpha_i,N_i)$ be a semi-fast convergent sequence satisfying $\alpha_i>(N_{i+1}+1)\alpha_{i+1}$. If $E(y_n)=E(\alpha_i,N_i)$ then for any $i\in\mathbb{N}$ the multiples of $\alpha_i$ in the sequence $(y_n)$ has a sum $N_i\alpha_i$ and form a complete sequence after scaled by $\alpha_i$.
\begin{proof}
Note that 
$$\alpha_1>(N_{2}+1)\alpha_{2},$$ 
which means that 
$$\alpha_1-\sum_{i=2}^{\infty}N_i\alpha_i>\alpha_2-\sum_{i=3}^{\infty}N_i\alpha_i.$$
By the Third Gap Lemma we know that any gap which is the longest from the left has a length $\alpha_p-\sum_{i=p+1}N_i\alpha_i$ for some $p$.
Since $(\sum_{i=2}^{\infty}N_i\alpha_i,\alpha_1)$ is the longest gap from the left we obtain that $\alpha_1\in (y_n)$. 
Let $m:=\max\{n: y_n=\alpha_1\}$. Then $y_m=\alpha_1$. Since $(\sum_{n=m+1}^{\infty}y_n,y_m)$ is the longest gap from the left in $E(y_n)=E(\alpha_i,N_i)$ we obtain that $\sum_{n=m+1}^{\infty}y_n=\sum_{k=2}^{\infty}N_k\alpha_k$. Thus $E_m(y_n)=E\big((\alpha_i,N_i)_{i\geq 2}\big)$. 

We will show that $F_m(y_n)=\{0,\alpha_1,2\alpha_1,\ldots,N_1\alpha_1\}$. 
\\'$\supset$'. Clearly $0\in F_m(y_n)$ and $\alpha_1\in F_m(y_n)$. Suppose that $k\alpha_1\notin F_m(y_n)$ for some $k\in\{2,\ldots,N_1\}$.  But $k\alpha_1\in E(y_n)$, so the equality $k\alpha_1=f+g$ holds for some $f\in F_m(y_n)$ and $g\in E_m(y_n)$. 
Since $E_m(y_n)\subset [0,\sum_{i=2}^{\infty}N_i\alpha_i]$ we obtain that $f\in (k\alpha_1-\sum_{i=2}^{\infty}N_i\alpha_i,k\alpha_1)$. But then the set $f+E_m(y_n)$, which is $\alpha_2- \sum_{i=3}^{\infty}N_i\alpha_i$-net as well as $E_m(y_n)$ cuts the gap o maximal length $\big((k-1)\alpha_1+\sum_{i=2}^{\infty}N_i\alpha_i,k\alpha_1\big)$, which yields a contradiction. 
\\'$\subset$'. Now we prove the reverse inclusion. Suppose that it does not hold  and take $f\in F_m(y_n)\setminus \{0,\alpha_1,\ldots,N_1\alpha_1\}$. Clearly $f\leq N_1\alpha_1$, because otherwise we immediately obtain a contradiction by the inequality
$$f+\sum_{i=2}^{\infty}N_i\alpha_i>\sum_{i=1}^{\infty}N_i\alpha_i.$$
Then $f+E_m(y_n)$ has a non-empty intersection with the gap of maximal length $(\lfloor\frac{f}{\alpha_1}\rfloor\alpha_1+\sum_{i=2}^{\infty}N_i\alpha_i,\lceil\frac{f}{\alpha_1}\rceil \alpha_1)$, a contradiction.
Hence the equality $$F_m(y_n)=\{0,\alpha_1,2\alpha_1,\ldots,N_1\alpha_1\}$$
holds, which means that $(y_n)_{n=1}^{m}$ is a complete sequence of multiples of $\alpha_1$. The proof for the next segments of elements $\alpha_n$'s is analogous.
\end{proof}
\end{thm}

By Theorem \ref{niemainnychrozwiniec} we know that the problem of  calculating the numbers of  semi-fast convergent representations reduces to the counting the number of representations of specific complete finite sequences. 
If a non-decreasing finite sequence $(z_n)_{n=1}^{m}$ is complete, that is $$E\big((z_n)_{n=1}^{m}\big)=\{0,1,2,\ldots,\sum_{n=1}^{m}z_n\}$$ then by Brown characterization \cite{Brown} the inequality $z_p\leq\sum_{n=1}^{p-1}z_n+1$ holds for all $p\leq m$. Conversely, if we treat the reverse sequence $(z_m,z_{m-1},\ldots,z_1)$ after possible scalling as some segment of multiples of some term in a semi-fast convergent sequence then if $z_p\neq z_{p-1}$ the value $z_{p}$ need to be greater than $\sum_{n=1}^{p-1}z_n$ and the part connected with the scalled tail. However $z_1=1$ is greater than the tail, so if we assume $z_p\geq\sum_{n=1}^{p-1}z_n+1$ then the semi-fast convergence holds. Hence if $z_p$ is the last in the segment of constant elements then the equality   $z_p=\sum_{n=1}^{p-1}z_n+1$ holds. 

Let us analyze the case in examples. We consider the complete finite sequences, which generate the given sets. We reverse them as they are placed in the semi-fast convergent sequence. It appears that not every reversal can be used, so we divide the family of complete sequences into two subfamilies: that which form a segment of multiples of some values in  a semi-fast convergent sequence and that which are not. For larger sets $\{1,2,\ldots,n\}$ for $n\geq 8$ we are interested only on the first subfamily, so we omit the second (also because there are too many of them).

\begin{center}
\begin{tabular}{|c||c|c|}
\hline
\hline
Set  & Semi-fast generator& Others \\
\hline
\hline
$\{0,1,2,3\}$ & $(1,1,1)$, $(2,1)$ & \\
\hline
$\{0,1,2,3,4\}$ & $(1,1,1,1)$ & $(2,1,1)$ \\
\hline
$\{0,1,2,3,4,5\}$ & $(1,\ldots,1)$, $(2,2,1)$, $(3,1,1)$ & $(2,1,1,1)$ \\
\hline
$\{0,1,2,3,4,5,6\}$ & $(1,\ldots,1)$ & $(2,1,1,1,1)$, $(2,2,1,1)$, \\
 &  & $(3,2,1)$, $(3,1,1,1)$ \\
\hline
$\{0,1,2,3,4,5,6,7\}$ & $(1,\ldots,1)$, $(2,2,2,1)$, &  $(2,1,\ldots,1)$,  $(2,2,1,1,1)$,  \\
 & $(4,2,1)$, $(4,1,1,1)$ &  $(3,2,1,1)$,  $(3,1,1,1,1)$, \\
\hline
$\{0,1,2,3,4,5,6,7,8\}$ & $(1,\ldots,1)$, $(3,3,1,1)$ &  \ldots  \\
\hline
$\{0,1,2,\ldots,9\}$ & $(1,\ldots,1)$, $(2,\ldots,2,1)$, &  \ldots  \\
 & $(5,1,\ldots,1)$, &    \\
\hline
$\{0,1,2,\ldots,10\}$ & $(1,\ldots,1)$ &  \ldots  \\
\hline
$\{0,1,2,\ldots,11\}$ & $(1,\ldots,1)$, $(2,\ldots,2,1)$ &  \ldots  \\
 & $(3,3,3,1,1)$, $(4,4,1,1,1)$ &    \\
  & $(6,1,\ldots,1)$, $(6,2,2,1)$, &    \\
    & $(6,3,1,1)$ &    \\
\hline
\end{tabular}
\end{center}

As we can see some sets $\{0,1,2,\ldots,n\}$ has unique constant semi-fast representation, for intance we can take $n\in\{4,6,10\}$, while for $n=11$ we have seven different representations. Let us denote by $R(n)$ the number of semi-fast representations for the set $\{0,1,2,\ldots,n\}$ for a given $n$. That is $R(4)=R(6)=R(10)=1$ and $R(11)=7$. Clearly the function $R(n)$ is far from being monotonous and our purpose is to obtain the formula for the values of the function $R$.  

We will show that $R(2^k-1)=2^{k-1}$ for any $k$, which will be useful for the most interesting case of Central Cantor sets. We have already seen that for $k\in\{1,2,3\}$, because $R(1)=1$, $R(3)=2$ and $R(7)=4$.   
First we show that if $(z_n)_{n=1}^{m}$ is a semi-fast, complete representation of $\{0,1,2,\ldots,2^{k}-1\}$ then the values of its terms are only among the powers of two. 

\begin{lem}\label{poweroftwo}
Let $k\in\mathbb{N}$. If  $(z_n)_{n=1}^{m}$ is a complete, semi-fast representation of the set $\{0,1,2,\ldots,2^{k}-1\}$ then for any $i\in\{1,\ldots,m\}$ there exists $0\leq p<m$ such that $z_i=2^p$.
\begin{proof}
Let $r$ be the number of repetitions of the maximal term in $(z_n)_{n=1}^{m}$, that is $z_m=z_{m-1}=\ldots=z_{m-r+1}> z_{m-r}$.
Then $z_m-1=\sum_{n=1}^{m-r}z_n$ and thus we obtain 
$$2^k-1=\sum_{n=1}^{m}z_n= z_m-1 + rz_m=(r+1)z_m-1,$$
so $(r+1)z_m=2^k$. Hence both $z_m$ and $(r+1)$ are the powers of $two$.  
We have already proved the assertion for all $i\geq m-r+1$. To finish the proof it is enough to note that the subsequence of initial terms  $(z_n)_{n=1}^{m-r}$ is complete and semi-fast convergent with a sum $z_m-1$. Since $z_m$ is a power of two the proof is inductive for $i< m-r$. 
\end{proof} 
\end{lem} 

\begin{cor}\label{numberofrepetition}
Note that Lemma \ref{poweroftwo} also determines the number of repetition of the chosen maximal element. If $z_m=2^p$ then $r=2^{k-p}-1$. 
\end{cor}
\begin{exa}
 Let us illustrate the result from Corollary \ref{numberofrepetition} for $k=3$. \\
For $p=2$ we have the maximal element $2^p=4$ of the representations $(4,2,1)$ and $(4,1,1,1)$ repeated $r=2^{k-p}-1=1$ time. \\
For $p=1$ we have the maximal element $2^p=2$ of $(2,2,2,1)$ repeated $r=2^{k-p}-1=3$-times.
\\
For $p=0$ we have the maximal element $2^p=1$ of $(1,\ldots,1)$ repeated $r=2^{k-p}-1=7$-times.
\end{exa}

\begin{thm}\label{potegidwojek}
The equality $R(2^k-1)=2^{k-1}$ holds for each $k\in\mathbb{N}$.
\begin{proof}
The assertion is clear for $k\in\{1,2,3\}$. We will prove it inductively. Suppose that the equality $R(2^n-1)=2^{n-1}$ holds for every $n<k$. We count the summands of $R(2^{k}-1)$ due to the maximal element in the representation. 
Let it have the value $2^p$ for some $1\leq p<k$. Then by Corollary \ref{numberofrepetition} it is repeated exactly $r=2^{k-p}-1$-times, that is $z_m=z_{m-1}=\ldots=z_{m-r+1}=2^p$. Hence $\sum_{i=1}^{m-r}z_i=2^p-1$. The subsequence $(z_n)$ of the any initial terms of the complete, semi-fast representation $(z_n)_{n=1}^{m}$ is also semi-fast convergent and complete. Hence $(z_n)_{n=1}^{m-r}$ is complete and semi-fast convergent. Thus the number of representations of $\{0,1,\ldots,2^{k}-1\}$ with the maximal term equal to $2^p$ is the same as the number of representations of  $\{0,1,\ldots,2^{p}-1\}$. By the inductive assumpiton we know that $R(2^p-1)=2^{p-1}$. Note also that for $p=0$ there is unique constant, complete, semi-fast convergent sequence of $1$'s, which represents the set $\{0,1,\ldots,2^{k}-1\}$. Hence we obtain 
$$R(2^{k}-1)=1+\sum_{p=1}^{k-1}R(2^{p}-1)=1+\sum_{p=1}^{k-1}2^{p-1}=1+2^{k-1}-1=2^{k-1},$$
which finishes the proof.
\end{proof}
\end{thm}

Combing the results from Theorems \ref{niemainnychrozwiniec} and \ref{potegidwojek} we obtain the formula for the numbers of semi-fast represetations of Central Cantor sets. 

\begin{cor}
Let $E:=E(\alpha_i,2^{M_i}-1)$ be a Central Cantor set and $\alpha_i>2^{M_{i+1}}\alpha_{i+1}$ for each $i$. 
Then the number of semi-fast convergent representations of $E$ is equal to 
\begin{itemize}
\item $\mathfrak{c}$ if and only if infinite many of $M_i$'s are greater than $1$;
\item $\prod_{i=1}^{\infty} 2^{M_i-1}$ if $M_i=1$ for all but finitely many $i$'s.
\end{itemize}
\end{cor}

Now we consider the values of the function $R$ on the whole $\mathbb{N}$ . Using the notation from Lemma \ref{poweroftwo} when we consider the set $\{0,1,2,\ldots,n\}$ we ask about the number of solutions of the equation 
$$n+1=(r+1)z_m$$
Hence we may generalize the result from Theorem \ref{potegidwojek} as follows
\begin{thm}\label{generalformula}
Let $n\in\mathbb{N}$ and $D_{n+1}=\{d_1,d_2,\ldots,d_t\}$ be the set of all proper divisors of $n+1$.
Then the number of complete, semi-fast convergent representations of $\{0,1,2,\ldots,n\}$ can be calculated by the formula
$$R(n)=\sum_{i=1}^{t} R(d_i-1).$$
\begin{proof}
Let $(z_k)$ be a complete, semi-fast representation of $\{0,1,2,\ldots,n\}$. 
By Lemma \ref{poweroftwo} the maximal value in the representation $z_m$ need to be a divisor of $n+1$. If $z_m=n+1$ then we take this element $r=0$ times, a contradiction, so $z_m$ need to be the proper divisor of $n+1$. Let $z_m=d_i$ for some $i\in\{1,\ldots,t\}$. Then after removing all $r$ maximal elements from the sequence $(z_k)$ the remaining sequence forms a complete, semi-fast convergent representation of $\{0,1,2,\ldots,z_m-1\}$. The number of all such representations is equal to $R(z_m-1)=R(d_i-1)$. We need to sum up all representations by considering all possible values for its maximal term. Note that if $d_i=1$ then the representation is a unique constant sequence of $1$'s, so we put $R(0)=1$. Thus we obtain the formula $R(n)=\sum_{i=1}^{t} R(d_i-1)$.
\end{proof}
\end{thm}

\begin{exa}
We show how to calculate the number of semi-fast convergent, complete representations of the set $\{0,1,2,\ldots,34\}$, that is the value $R(34)$. Let us consider the set of all proper divisors of $35$, that is $D_{35}=\{1,5,7\}$. Hence by the formula given in Theorem \ref{generalformula} we obtain 
$$R(34)=R(0)+R(4)+R(6)=1+1+1=3.$$
\end{exa}

\begin{cor}
Let $n\in\mathbb{N}$. The equality $R(n)=1$ holds if and only if $n+1$ is a prime number. Indeed by the formula in Theorem \ref{generalformula} the number $n+1$ should have only one proper divisor, that is $D_{n+1}=\{1\}$, which is equivalent to that $n+1$ is a prime number. 
\end{cor}

\begin{cor}
Let $(\alpha_i,{N_i})$ be a semi-fast convergent sequence and $\alpha_i>(N_{i+1}+1)\alpha_{i+1}$ for each $i$. 
Then the number of semi-fast convergent representations of the achievement set $E(\alpha_i,{N_i})$ is equal to 
\begin{itemize}
\item $\mathfrak{c}$ if and only if infinite many of $N_i+1$ are not prime numbers;
\item $\prod_{i=1}^{\infty} R(N_i)$ if all but finitely many of $N_i+1$ are prime numbers;
\end{itemize}
In particular the set of subsums  $E(\alpha_i,{N_i})$ has a unique semi-fast representation given by constant segments of $\alpha_i$'s if and only if  $N_i+1$ is a prime number for each $i$. 
\end{cor}

Now we want to calculate the number of all representations of $E(\alpha_i,{N_i})$ satisfying assumption $\alpha_i>(N_{i+1}+1)\alpha_{i+1}$ for each $i$. It means that we count any $(y_n)$ such that $E(y_n)=E(\alpha_i,{N_i})$ with no longer assuming that the sequence $(y_n)$ is semi-fast convergent. By Theorem \ref{niemainnychrozwiniec} we need to find the number of all complete integer partitions of the given natural number $n$ - we denote it by $C(n)$. If we sum up all of the representations from both columns ''Semi-fast generator'' and ''Others'' in the Table we see that $C(3)=2$, $C(4)=2$, $C(5)=4$, $C(6)=5$, $C(7)=8$. We also have $C(1)=C(2)=1$.  

The reccurence formula for the values of the function $C(n)$ is given in \cite{Park}. Hence we obtain 
\begin{cor}
Let $(\alpha_i,{N_i})$ be a semi-fast convergent sequence and $\alpha_i>(N_{i+1}+1)\alpha_{i+1}$ for each $i$. 
Then the number of all representations of the achievement set $E(\alpha_i,{N_i})$ is equal to 
\begin{itemize}
\item $\mathfrak{c}$ if and only if infinite many of $N_i$ are greater than $2$;
\item $\prod_{i=1}^{\infty} C(N_i)$ if and only if all but finitely many of $N_i$ are equal to $1$ either $2$;
\end{itemize}
\end{cor}

 \section*{Funding Declaration} 
This research received no external funding. 

\section*{Conflict of Interest} I confirm that there is no conflict of interest.

 \section*{Data Availability} Data sharing is not applicable to this article as no datasets were generated or analyzed during the current study.

\end{document}